\documentclass[10pt]{amsart}
\usepackage[T1]{fontenc}
\usepackage[utf8]{inputenc}
\usepackage[margin=1in]{geometry}
\usepackage{amsmath,amssymb,amsthm,mathtools,booktabs,microtype}
\usepackage[hidelinks]{hyperref}
\newtheorem{theorem}{Theorem}[section]
\newtheorem{lemma}[theorem]{Lemma}
\newtheorem{proposition}[theorem]{Proposition}
\newtheorem{corollary}[theorem]{Corollary}
\theoremstyle{remark}
\newtheorem{remark}[theorem]{Remark}
\newcommand{\R}{\mathbb R}
\newcommand{\Z}{\mathbb Z}

\newcommand{\DotSet}{\Lambda}
\title[Dot products and separate growth-minimizing bases]
{A dot-product bound from separate growth-minimizing bases}
\author{Zhipeng Lu}
\address{Shenzhen MSU--BIT University and Guangdong Laboratory of
Machine Perception and Intelligent Computing, Shenzhen 518172,
Guangdong, China}
\email{zhipeng.lu@hotmail.com}
\date{}
\begin{document}
\begin{abstract}
For every finite $P\subset\R^2$ we prove
$|\{p\cdot q:p,q\in P\}|\gg |P|^{199/295}$, with an absolute constant
and no logarithmic loss, where $199/295=2/3+7/885$. This improves the
bound $2/3+7/1425$ of Kokkinos, which in turn had improved the first
superthreshold bound of Hanson, Roche-Newton, and Senger. The main new
ingredient is the inequality
$|F^{(2)}G^{(2)}/(F^{(2)}G)|
\le |AB|(|AF|/|A|)^4(|BG|/|B|)^3$:
each of the two radial profiles retains its own Petridis
growth-minimizing subset, and the product of the two subsets serves as
a common base for both growth operators. A prime-power construction
shows this inequality is sharp under its hypotheses. The second
ingredient is a weighted-median replacement for the dyadic selection in
the squeezing argument of Roche-Newton and Wong, which upgrades their
seven-factor expander to a logarithm-free form; an appendix gives the
complete proof from the Solymosi--Zahl incidence theorem. We complement
the lower bound with a structural third-moment estimate for
near-extremal configurations and with a construction showing that,
in contrast, the \emph{pinned} problem admits no linear lower bound:
for every $\varepsilon>0$ there are arbitrarily large $P$ with
$\max_{p\in P}|p\cdot P|\le\varepsilon|P|$, and indeed
$\max_{p}|p\cdot P|=O(|P|/\sqrt{\log|P|})$ is attainable.
\end{abstract}
\maketitle

\section{Introduction}

For a finite set $P\subset\R^2$, put
\[
 \DotSet(P)=\{p\cdot q:p,q\in P\},\qquad N=|P|,\qquad t=|\DotSet(P)|.
\]
How small can $t$ be? The question is a natural companion of the
Erd\H{o}s distinct-distance problem, resolved up to logarithms by Guth
and Katz \cite{GK}, but it behaves differently: a geometric progression
on a single ray already achieves $t=2N-1$, so the extremal behavior is
governed by configurations spread over many directions. The
Szemer\'edi--Trotter theorem \cite{ST} gives $t\gg N^{2/3}$
(see \cite{HRS}), and this threshold stood until Hanson, Roche-Newton,
and Senger \cite{HRS} proved $t\gtrsim N^{2/3+c}$ for an unspecified
absolute $c>0$, by introducing a superquadratic expander for products
of shifts. Kokkinos \cite{Kok} made the gain explicit, proving
$t\gtrsim N^{2/3+7/1425}$ using the stronger expander of Roche-Newton
and Wong \cite{RNW}. Here $X\gtrsim Y$ allows a factor
$(\log(2Y))^{-C}$, while $X\gg Y$ means $X\ge cY$ with $c>0$ absolute.

\begin{theorem}\label{thm:main}
For every finite $P\subset\R^2$ with $|P|\ge2$,
\begin{equation}\label{eq:main}
 |\DotSet(P)|\gg |P|^{199/295}
     =|P|^{2/3+7/885}.
\end{equation}
\end{theorem}

Self-pairs are included in $\DotSet(P)$; Section \ref{sec:distinct}
gives the same exponent for distinct pairs. Two independent
improvements over \cite{Kok} enter, one in the exponent and one in the
logarithms.

\smallskip\noindent\textbf{The exponent.}\enspace
All arguments since \cite{HRS} feed two radial profiles $A,B$ and two
shifted slope sets $F,G$, tied by three product budgets
$|AF|,|BG|,|AB|\le t$, into a seven-factor expander. The existing
route passes both profiles through a common dilated intersection of
guaranteed size $|A||B|/|AB|$; with profile sizes at least $m$, this
costs five extra factors of $t/m$ (Remark \ref{rem:saving}). We avoid the intersection entirely:
by Petridis's minimal-growth lemma, each profile contains its own
subset minimizing expansion by its own shift set, and, because the
lemma holds for an arbitrary auxiliary set, the \emph{product} of the
two minimizers is a common base for both growth operators. This yields
the key inequality
\begin{equation*}
 \left|\frac{F^{(2)}G^{(2)}}{F^{(2)}G}\right|
 \le |AB|
      \left(\frac{|AF|}{|A|}\right)^4
      \left(\frac{|BG|}{|B|}\right)^3,
\end{equation*}
proved as Corollary \ref{cor:bridge}, and sharp under exactly these
hypotheses by a prime-power construction
(Proposition \ref{prop:sharp}). Already with the published expander of
\cite{RNW} this gives $2/3+7/885$; the improvement over $2/3+7/1425$
is entirely in this step.

\smallskip\noindent\textbf{The logarithms.}\enspace
The expander of \cite{RNW}, like its predecessor in \cite{HRS},
carries logarithmic losses from a dyadic pigeonholing in its squeezing
step. We replace that selection by a weighted median
(Lemma \ref{lem:median}), which retains half the indices while
controlling the gap palette exactly, and we verify the associated
algebraic surface directly in quadratic coordinates
(Lemma \ref{lem:quadric}). The result is a logarithm-free form of
their theorem.

\begin{theorem}[Logarithm-free RNW expander]\label{thm:expander}
For every finite nonempty $S\subset\R$ with $1+SS\subset\R_{>0}$,
there exist $s,s'\in S$ such that
\begin{equation}\label{eq:expander}
 \left|
 \frac{(1+sS)^{(2)}(1+s'S)^{(2)}}
      {(1+sS)^{(2)}(1+s'S)}
 \right|\gg |S|^{31/12}.
\end{equation}
\end{theorem}

Here $F^{(j)}$ denotes the $j$-fold product set, with independent
choices of the factors, $F^{(0)}=\{1\}$, and quotient sets likewise
have independent numerator and denominator choices. All displayed
factors are positive, so no zero denominators occur. With $\gtrsim$ in
place of $\gg$, Theorem \ref{thm:expander} is a special case of
\cite[Theorem 1]{RNW} (Theorem 1.1 of its arXiv version), which is
also the input of \cite[Theorem 4.1]{Kok}; the exponent $31/12$
improved the exponent $5/2$ of \cite{HRS}. Appendix
\ref{app:expander} proves the stronger form stated here; its only
incidence input is the Solymosi--Zahl theorem \cite[Theorem 1.4]{SZ},
whose constant depends only on the degree.

Section \ref{sec:profiles} adds structural information: a third-moment
bound on the ray occupancies which shows that any configuration within
a constant of the exponent $199/295$ must use its directions, its
occupancies, and every pinned palette at full strength. Section
\ref{sec:pinned} shows by construction that this saturation is not
automatic but extremal: the pinned quantity
$\max_{p\in P}|p\cdot P|$ admits no linear lower bound, and can be as
small as $O(N/\sqrt{\log N})$, while known bounds keep it
$\gg N^{2/3}$. To the best of our knowledge, \cite{Kok} was the
strongest previously recorded global bound, and no sublinear pinned
construction was known.\footnote{Propositions
\ref{prop:multibase} and \ref{prop:pinned}, Corollary \ref{cor:bridge},
and the identities of Lemmas \ref{lem:median} and \ref{lem:quadric}
have been formally verified in Lean~4 over \texttt{mathlib}; the
verification files accompany this article as ancillary files.}

\section{Separate minimizing subsets can share a product base}

All sets in this section are finite nonempty subsets of a commutative
group, written multiplicatively. We give the needed growth argument
in full.

\begin{lemma}[Ruzsa quotient inequality]\label{lem:ruzsa}
For finite nonempty $C,U,V$,
\[
 |U/V|\,|C|\le |CU|\,|CV|.
\]
\end{lemma}
\begin{proof}
Choose one representation $w=u_w/v_w$ for each $w\in U/V$.
The map $(w,c)\mapsto(cu_w,cv_w)$ is injective into $CU\times CV$:
the quotient recovers $w$, and then the first coordinate recovers $c$.
\end{proof}

\begin{lemma}[Petridis's minimal-growth lemma]\label{lem:petridis}
Let $X,F$ be finite nonempty sets, and put $\alpha=|XF|/|X|$.
Suppose
\[
 |ZF|\ge\alpha|Z|\qquad\text{for every }Z\subseteq X.
\]
Then, for every finite nonempty $C$,
\begin{equation}\label{eq:petridis}
 |XFC|\le\alpha|XC|.
\end{equation}
\end{lemma}
\begin{proof}
Write $C=\{c_1,\ldots,c_r\}$ and put
\[
 X_i=\{x\in X:c_i x\notin\bigcup_{j<i}c_jX\}.
\]
The sets $c_iX_i$ partition $CX$, so $|CX|=\sum_i|X_i|$.
Every element of $c_i(X\setminus X_i)$ is in
$\bigcup_{j<i}c_jX$. Multiplication by $F$ therefore shows that the
new contribution of $c_iXF$ to $\bigcup_{j\le i}c_jXF$ has size at most
\[
 |XF|-|(X\setminus X_i)F|
 \le\alpha|X|-\alpha|X\setminus X_i|
 =\alpha|X_i|.
\]
Summing proves \eqref{eq:petridis}.
This is the argument of \cite[Proposition 2.1]{Pet}.
\end{proof}

\begin{proposition}[Several growth bases]\label{prop:multibase}
Let $A_1,\ldots,A_r$ and $F_1,\ldots,F_r$ be finite nonempty sets.
For arbitrary nonnegative integers $p_i,q_i$,
\begin{equation}\label{eq:multibase}
 \left|\frac{\prod_{i=1}^r F_i^{(p_i)}}
                 {\prod_{i=1}^r F_i^{(q_i)}}\right|
 \le
 \left|\prod_{i=1}^r A_i\right|
 \prod_{i=1}^r
       \left(\frac{|A_iF_i|}{|A_i|}\right)^{p_i+q_i}.
\end{equation}
\end{proposition}
\begin{proof}
For each $i$, choose a nonempty $X_i\subseteq A_i$ minimizing
$\alpha_i=|X_iF_i|/|X_i|$ over all nonempty subsets of $A_i$.
Then $\alpha_i\le |A_iF_i|/|A_i|$, and Lemma \ref{lem:petridis}
applies to $X_i,F_i$ with every auxiliary set $C$.
Put $C_0=X_1\cdots X_r$. Repeated application of the lemma gives
\begin{equation}\label{eq:jointgrowth}
 \left|C_0\prod_i F_i^{(k_i)}\right|
       \le |C_0|\prod_i\alpha_i^{k_i}
       \qquad(k_i\ge0).
\end{equation}
For example one can first remove all copies of $F_1$, using
$X_2\cdots X_r$ and the other factors as the auxiliary set, and then
proceed successively through the remaining $F_i$.

Apply Lemma \ref{lem:ruzsa} with the common set $C_0$, numerator
$U=\prod_iF_i^{(p_i)}$, and denominator
$V=\prod_iF_i^{(q_i)}$. Equation \eqref{eq:jointgrowth} gives
\[
 |U/V|\le |C_0|\prod_i\alpha_i^{p_i+q_i}.
\]
Finally $C_0\subseteq A_1\cdots A_r$.
No lower bound on the sizes of the minimizing subsets is needed.
\end{proof}

\begin{corollary}[The seven-factor estimate]\label{cor:bridge}
For finite nonempty $A,B,F,G$,
\begin{equation}\label{eq:bridge}
 \left|\frac{F^{(2)}G^{(2)}}{F^{(2)}G}\right|
 \le |AB|
      \left(\frac{|AF|}{|A|}\right)^4
      \left(\frac{|BG|}{|B|}\right)^3.
\end{equation}
In particular, if $|AF|,|BG|,|AB|\le t$ and $|A|,|B|\ge m$, then
\begin{equation}\label{eq:bridge-m}
 \left|\frac{F^{(2)}G^{(2)}}{F^{(2)}G}\right|
 \le \frac{t^8}{|A|^4|B|^3}
 \le \frac{t^8}{m^7}.
\end{equation}
\end{corollary}
\begin{proof}
Use Proposition \ref{prop:multibase} with two bases and exponent pairs
$(p_1,q_1)=(2,2)$ and $(p_2,q_2)=(2,1)$.
\end{proof}

\begin{remark}[Where the saving occurs]\label{rem:saving}
The common-intersection method chooses a dilate with
$|A\cap xB|\ge |A||B|/|AB|\ge m^2/t$ and then applies growth estimates
from that common subset; with the same budgets this gives
$\ll t^{13}/m^{12}$ for the seven-factor set, and, through the
geometry of Section \ref{sec:rays}, exactly the exponent $319/475$ of
\cite{Kok}. Equation \eqref{eq:bridge-m} instead gives $t^8/m^7$: the
product $X_1X_2$ of the two minimizing subsets is the common base, and
its upper bound $|X_1X_2|\le|AB|$ is all that is required.
\end{remark}

\begin{proposition}[Sharpness of the seven-factor estimate]
\label{prop:sharp}
For every $L\ge4$ there are finite sets $A,B,F,G\subset\mathbb Q_{>0}$
with $|A|=|B|=m=L^2$ and $|AF|,|BG|,|AB|\le t=4L^3$ such that
\[
 \left|\frac{F^{(2)}G^{(2)}}{F^{(2)}G}\right|
 =\Theta\bigl(m(t/m)^8\bigr).
\]
Thus the exponent $8$ in \eqref{eq:bridge-m} cannot be lowered using
only the stated cardinality budgets.
\end{proposition}
\begin{proof}
Let $q_1,\ldots,q_L,r_1,\ldots,r_L$ be distinct primes other than
$2,3,5$, and set
\[
 A=\{2^i3^j:0\le i,j<L\},\quad
 B=\{2^i5^j:0\le i,j<L\},\quad
 F=A\{q_1,\ldots,q_L\},\quad
 G=B\{r_1,\ldots,r_L\}.
\]
Unique factorization gives $|A|=|B|=L^2$,
$|AF|=|BG|=(2L-1)^2L\le t$, and $|AB|=(2L-1)L^2\le t$. Since $F$ and
$G$ are full Cartesian products of their prime coordinates, the
elements of $W=F^{(2)}G^{(2)}/(F^{(2)}G)$ realize all combinations of:
a $2$-adic valuation in an interval of length $7L-6$, a $3$-adic one
of length $4L-3$, a $5$-adic one of length $3L-2$, at least
$\binom L2\binom{L-2}2$ sign patterns in the $q_\nu$ (two distinct
positive and two disjoint negative indices), and at least
$\binom L2(L-2)$ patterns in the $r_\nu$. Hence $|W|\gg L^{10}$.
Conversely the three small-prime valuations take $O(L)$ values each
and the $q$- and $r$-patterns at most $L^4$ and $L^3$, so
$|W|=\Theta(L^{10})=\Theta(m(t/m)^8)$, as $t/m=4L$.
\end{proof}

Improvements on Theorem \ref{thm:main} through this bridge must
therefore use more than the three budgets --- for instance the affine
structure $F=1+sS$, $G=1+s'S$ present in the application.

\section{Rich rays and one occupied projection fiber}\label{sec:rays}

An origin-preserving rotation and a nonzero uniform dilation preserve
$N,t$, and the occupied radial directions. Translation is not used.
After deleting the origin and restricting to one of a fixed number of
angular sectors, it suffices to consider $P$ contained in a sector of
angle less than $\pi/2$. All its dot products are positive.
Let $D$ be the number of occupied rays.

\begin{lemma}[Directional incidence bound]\label{lem:direction}
For such a set,
\begin{equation}\label{eq:direction}
 t\gg\sqrt{ND}.
\end{equation}
\end{lemma}
\begin{proof}
For each occupied ray choose a pin $p_i\in P$. Its occupied level lines
$p_i\cdot x=z$, $z\in p_i\cdot P$, form a family of at most $t$ lines.
Families belonging to distinct rays have distinct directions, so their
union consists of $L\le Dt$ distinct lines and has exactly $ND$
incidences with $P$. Szemer\'edi--Trotter gives
\[
 ND\le C\bigl(N^{2/3}L^{2/3}+N+L\bigr),
\]
so at least one term on the right is at least $ND/(3C)$.
If $ND\le3CN$, then $D\le3C$, and the heaviest ray carries at least
$N/(3C)$ points, which its own pin maps injectively into the palette;
hence $t\ge N/(3C)$.
If $ND\le3CL\le3CDt$, then again $t\ge N/(3C)$.
If $ND\le3CN^{2/3}(Dt)^{2/3}$, cubing gives $ND\le27C^3t^2$.
Since $D\le N$, each case yields $t\gg\sqrt{ND}$.
This is the standard pinned directional estimate, also recorded in
\cite[Lemma 4.2]{Kok}.
\end{proof}

\begin{lemma}[Uniform lower occupancy without dyadic loss]\label{lem:slice}
There is a choice of coordinates, obtained by rotation and dilation,
and a finite slope set $S$ such that
\begin{enumerate}
\item $(1,s)\in P$ for every $s\in S$;
\item $|S|\ge N/(2t)$;
\item the full radial coordinate sets
\[
 A_s=\{x>0:(x,sx)\in P\}
\]
satisfy $|A_s|\ge m_0=N/(2D)$ for every $s\in S$.
\end{enumerate}
\end{lemma}
\begin{proof}
Call a ray rich if it contains at least $m_0=N/(2D)$ points.
Nonrich rays together contain fewer than $N/2$ points, so the union
$Q$ of rich rays contains at least $N/2$ points.

Rotate an actual point of $P$ to the positive horizontal axis. Since
all points lie in an acute sector with that point, every point has
positive horizontal coordinate. There are at most $t$ distinct
horizontal coordinates of $Q$, because they are a fixed scalar
multiple of the dot products with that actual pin.
Some vertical line $x=x_0>0$ therefore contains at least $N/(2t)$
points of $Q$. It meets every occupied ray at most once.
Dilate the whole configuration by $1/x_0$ and let $S$ be the slopes
of these points. Every one of their rays is rich.
\end{proof}

\begin{remark}
Lemma \ref{lem:slice} uses the full radial sets on the selected rays.
It neither requires comparable upper occupancies nor extracts a common
radial profile. From \eqref{eq:direction},
$m_0=N/(2D)\gg N^2/t^2$.
\end{remark}

\section{Proof of the dot-product bound}

Apply Lemma \ref{lem:slice}. If $|S|$ is bounded, its lower bound
already gives $t\gg N$, so suppose $S$ is large.
Choose $s,s'\in S$ using Theorem \ref{thm:expander}, and put
\[
 A=A_s,\qquad B=A_{s'},\qquad
 F=1+sS,\qquad G=1+s'S.
\]
All these sets are positive. The point identities
\begin{align*}
 (a,sa)\cdot(1,u)&=a(1+su),\\
 (b,s'b)\cdot(1,u)&=b(1+s'u),\\
 (a,sa)\cdot(b,s'b)&=ab(1+ss')
\end{align*}
show respectively that
\begin{equation}\label{eq:threebudgets}
 |AF|\le t,\qquad |BG|\le t,\qquad |AB|\le t.
\end{equation}
The last inequality uses $1+ss'>0$, so the angular coefficient only
rescales the product set. These statements also hold if $s=s'$.

Corollary \ref{cor:bridge} now yields
\[
 |S|^{31/12}\ll
 \left|\frac{F^{(2)}G^{(2)}}{F^{(2)}G}\right|
 \le\frac{t^8}{m_0^7}
 \ll\frac{t^{22}}{N^{14}}.
\]
Together with $|S|\ge N/(2t)$, this gives
\[
 \left(\frac Nt\right)^{31/12}
   \ll\frac{t^{22}}{N^{14}},
 \qquad
 N^{199/12}\ll t^{295/12}.
\]
This proves Theorem \ref{thm:main}.
The fixed-density angular restriction only changes absolute constants.
There is no logarithmic loss, by Appendix \ref{app:expander}.
\qed

\begin{corollary}[Transfer of a better seven-factor expander]
\label{cor:beta}
If \eqref{eq:expander} holds with exponent $\beta>2$ in place of
$31/12$, then
\begin{equation}\label{eq:beta}
 t\gg N^{(14+\beta)/(22+\beta)}
      =N^{2/3+(\beta-2)/(3(22+\beta))}.
\end{equation}
\end{corollary}
\begin{proof}
Repeat the preceding calculation with $\beta$.
If the assumed expander instead has a polylogarithmic loss, the
conclusion has a polylogarithmic loss as well.
\end{proof}

\begin{remark}[Numerical comparison]
For the expander exponent $\beta=5/2$ of \cite{HRS}, the same growth
argument already gives $t\gg N^{33/49}=N^{2/3+1/147}$; for
$\beta=31/12$ it gives $199/295$, and
$199/295-319/475=84/28025>0$.
\end{remark}

\subsection{Distinct pairs}\label{sec:distinct}

Define $t_0=|\{p\cdot q:p,q\in P,\ p\ne q\}|$.
The same exponent holds with $t_0$ in place of $t$.
Indeed every full pinned palette has at most $t_0+1$ elements, and
each radial line contains at most $t_0+1$ points. Thus the directional
incidence bound and the rich-ray slice selection work with budget
$t_0+1$.

In \eqref{eq:threebudgets}, each of $AF$ and $BG$ has at most one
additional self-pair value, so its size is at most $t_0+1$.
If $s\ne s'$, the cross product $AB$ involves only distinct points.
If $s=s'$, the full $AA$ has at most $|A|\le t_0+1$ additional diagonal
values, so $|AA|\le2t_0+1$. Hence all three budgets are bounded by
$2t_0+1$, proving the same power bound.

The proof combines exactly three inputs: the incidence bound for the
number of directions, the rank-two dot-product identity for the three
product budgets, and the separate minimizing subsets that let those
budgets enter the expander simultaneously. No additive-energy or
unit-equation estimates are assumed.

\section{Using all projection fibers and radial occupancies}
\label{sec:profiles}

Here $P$ is a nonzero configuration contained in an acute sector.
Let its occupied rays have sizes $a_1,\ldots,a_D$, with
$\sum_i a_i=N$. For an actual pin $p\in P$ write
$\tau_p=|p\cdot P|$.

\begin{theorem}[Third moment of the radial profile]\label{thm:profile}
For every actual pin $p\in P$,
\begin{equation}\label{eq:thirdmoment}
 \sum_{i=1}^D a_i^3
 \ll \tau_p^{31/42}t^{16/7}N^{11/42}.
\end{equation}
In particular, putting
$H_2=D^2\sum_i a_i^3/N^3\ge1$ and $J_p=t/\tau_p\ge1$ gives
\begin{equation}\label{eq:profilebound}
 t\gg H_2^{42/295}J_p^{31/295}N^{199/295}.
\end{equation}
\end{theorem}

Since $H_2,J_p\ge1$, \eqref{eq:profilebound} contains
Theorem \ref{thm:main}; the proof uses only the fiber identities and
Lemma \ref{lem:direction}, so it is a second route to the main bound,
independent of Lemma \ref{lem:slice}.

\begin{proof}
On a nonempty level line $\ell_z=\{x:p\cdot x=z\}$ let
$k_z=|P\cap\ell_z|$ and
\[
 n_z(h)=|\{x\in P\cap\ell_z:
              a_{\operatorname{ray}(x)}\ge h\}|.
\]
Rotate and dilate this line to $x_1=1$.
The same three product budgets as in \eqref{eq:threebudgets},
applied only to its points counted by $n_z(h)$, and
Theorem \ref{thm:expander} imply
\[
 n_z(h)^{31/12}\ll t^8/h^7,\qquad
 n_z(h)\le\min(k_z,K h^{-q}),
 \quad q=\frac{84}{31},\quad K=Ct^{96/31}.
\]
The full radial sets supply the weights; no points on other fibers
are removed from those sets.
For $0<u<q$, the layer-cake identity and splitting at
$(K/k_z)^{1/q}$ give
\[
 \sum_{x\in P\cap\ell_z}
       a_{\operatorname{ray}(x)}^u
 \le\frac{q}{q-u}K^{u/q}k_z^{1-u/q}.
\]
Sum over the $\tau_p$ nonempty fibers and use
$\sum_z k_z^{1-u/q}\le\tau_p^{u/q}N^{1-u/q}$.
Each ray of size $a_i$ contributes $a_i^{u+1}$.
Taking $u=2$ proves \eqref{eq:thirdmoment}; the interpolation
constant is $q/(q-2)=42/11$, so there is no endpoint logarithm.
After taking 42nd powers, \eqref{eq:thirdmoment} becomes
$H_2^{42}N^{115}\ll D^{84}\tau_p^{31}t^{96}$.
Use $D\ll t^2/N$ from Lemma \ref{lem:direction} and
$\tau_p=t/J_p$ to obtain \eqref{eq:profilebound}.
\end{proof}

\begin{remark}[What a nearly extremal configuration must satisfy]
The more precise form, with $L=t^2/(ND)$, is
\[
 t\gg H_2^{42/295}L^{84/295}J_p^{31/295}N^{199/295}.
\]
Consequently, if $t=O(N^{199/295})$, the normalized radial third
moment is bounded, $D$ is comparable to $t^2/N$, and every actual
pin has $\tau_p$ comparable to $t$.
Thus a power-sized defect in any of these three requirements gives
a strict power improvement for that class of configurations.
Regular numerical profiles can saturate all three estimates;
this argument does not exclude that remaining case.
\end{remark}

\begin{proposition}[Additional gain from aligned radial overlap]
\label{prop:overlap}
In the rich-ray extraction, suppose the two slopes selected by
Theorem \ref{thm:expander} have radial profiles $A,B$ satisfying
$|A\cap\lambda B|\ge\theta m_0$ for some $\lambda>0$ and $\theta>0$.
Then
\begin{equation}\label{eq:overlap}
 t\gg\theta^{72/259}N^{175/259}.
\end{equation}
This is a conditional statement about the selected profiles.
\end{proposition}
\begin{proof}
Put $C=A\cap\lambda B$, $c=|C|$, and $H=F\cup G$.
The inclusions $CF\subseteq\DotSet(P)$ and
$CG\subseteq\lambda\DotSet(P)$ give $|CH|\le2t$.
The one-base case of Proposition \ref{prop:multibase} yields
\[
 |F^{(2)}G^{(2)}/(F^{(2)}G)|
 \le |H^{(4)}/H^{(3)}|
 \le c(2t/c)^7.
\]
Combining $c\ge\theta m_0\gg\theta N^2/t^2$ with the expander and
$|S|\gg N/t$ gives
$N^{175/12}\ll\theta^{-6}t^{259/12}$,
which is \eqref{eq:overlap}.
\end{proof}

\section{No linear bound for pinned dot products}\label{sec:pinned}

Theorem \ref{thm:profile} shows that near the global extremal every
pinned palette is as large as the whole palette. In general, however,
the pinned quantity admits no linear lower bound.

\begin{proposition}\label{prop:pinned}
For every integer $D\ge2$ there are arbitrarily large finite sets
$P\subset\R^2$, occupying $D$ rays, with
\[
 \max_{p\in P}\;|\{p\cdot q:q\in P\}|\ \le\ \frac{e\,|P|}{D}.
\]
In particular, for every $\varepsilon>0$ there are arbitrarily large
$P$ with $\max_{p\in P}|p\cdot P|\le\varepsilon|P|$.
\end{proposition}
\begin{proof}
Use the directions $(1,i)$, $1\le i\le D$. Let $p_1,\ldots,p_r$ be the
primes dividing some $1+ij$ with $1\le i,j\le D$, write
$1+ij=\prod_\nu p_\nu^{f_\nu(i,j)}$, and put
$H=\max_{i,j}\sum_\nu f_\nu(i,j)$. For $K\ge rH$ let
\[
 T_K=\Bigl\{e\in\Z_{\ge0}^r:\sum\nolimits_\nu e_\nu\le K\Bigr\},\qquad
 A_K=\Bigl\{\prod\nolimits_\nu p_\nu^{e_\nu}:e\in T_K\Bigr\},
\]
so that $|A_K|=|T_K|=\binom{K+r}{r}=:m$ by unique factorization.
Take $P=\{a(1,i):a\in A_K,\ 1\le i\le D\}$, so that $N=Dm$.
A pin $p=a(1,j)$ has
$p\cdot P=a\bigcup_{i\le D}(1+ij)A_K$, and the exponent set of
$(1+ij)A_K$ is $f(i,j)+T_K\subseteq T_{K+H}$. Hence
\[
 |p\cdot P|\le\binom{K+H+r}{r}
 =m\prod_{\ell=1}^{r}\frac{K+H+\ell}{K+\ell}
 \le m\Bigl(1+\frac{H}{K+1}\Bigr)^{r}
 \le m\,e^{rH/(K+1)}\le em=\frac{eN}{D}.
\]
Letting $K\to\infty$ gives arbitrarily large $N$; the second statement
follows with $D>e/\varepsilon$.
\end{proof}

\begin{remark}\label{rem:pinnedrange}
Every pin in the construction also satisfies $|p\cdot P|\ge m$, and
$A_KA_K$ has exponent set exactly $T_{2K}$, so the global palette pays
the full doubling $\binom{2K+r}{r}$: compressing every pin inflates
the global set, in exact agreement with Theorem \ref{thm:profile}.
Quantitatively, $r\le\pi(1+D^2)\ll D^2/\log D$ and
$H\le\log_2(1+D^2)$, so the choice $K=rH$ gives
$\log N\ll D^2\log\log D/\log D$, hence $D\gg\sqrt{\log N}$ and
\[
 \max_{p\in P}\,|p\cdot P|\ \ll\ \frac{N}{\sqrt{\log N}}.
\]
In the other direction, the pinned form of Lemma \ref{lem:direction}
(see \cite[Lemma 4.2]{Kok}) together with the heaviest radial line
gives $\max_p|p\cdot P|\gg\max(\sqrt{ND},N/D)\ge N^{2/3}$ for every
finite $P$; for further pinned estimates see \cite{AGHS,BMS}. Whether
$\max_p|p\cdot P|$ can be $O(N^{1-\delta})$ for some $\delta>0$
remains open.
\end{remark}

\appendix
\section{A logarithm-free reconstruction of the expander}
\label{app:expander}

The convexity and squeezing architecture in this appendix is due to
Roche-Newton--Wong \cite[Section 4]{RNW}. We replace the dyadic
selection by a weighted median and verify the algebraic surface
directly in quadratic coordinates.
The following incidence theorem is used as a published input:
if $Q\in\R[X,Y,Z]$ is irreducible over $\mathbb C$ of bounded degree, none of its
partial derivatives vanishes identically, and its surface is not
locally additive under separate complex analytic coordinate changes, then
\begin{equation}\label{eq:SZ}
 |Z(Q)\cap(U\times V\times W)|\ll M^{12/7},
 \qquad M=\max(|U|,|V|,|W|).
\end{equation}
This follows from \cite[Theorem 1.4]{SZ}, whose constant depends
only on the degree and has no logarithmic or epsilon loss.

\begin{lemma}[Median squeezing]\label{lem:median}
Suppose
$y_1<z_1<y_2<z_2<\cdots<y_n<z_n$, and the positive numbers
$x_j=z_j-y_j$ are strictly increasing.
Put $Y=\{y_j\}$ and $Z=\{z_j\}$.
For every nonempty
$I\subseteq\{1,\ldots,\lfloor n/2\rfloor-1\}$, there is
$I'\subseteq I$ with $|I'|\ge |I|/2$ such that
\begin{equation}\label{eq:median}
 |2Y+2Z-2Y-Z|
 \ge \frac{n|I|}{4}
       |\{x_{j+1}-x_j:j\in I'\}|.
\end{equation}
\end{lemma}
\begin{proof}
Write $\gamma_j=x_{j+1}-x_j>0$ and $\Gamma=\{\gamma_j:j\in I\}$, and for
$v\in\Gamma$ put $S_v=\{j\in I:\gamma_j\le v\}$. Let $\tau\in\Gamma$ be
minimal with $2|S_\tau|\ge|I|$; the largest element of $\Gamma$
qualifies, so $\tau$ exists. Take $I'=S_\tau$. Then $|I'|\ge|I|/2$, and
the gap set displayed in \eqref{eq:median} equals
$\{v\in\Gamma:v\le\tau\}$; let $q$ be its size.

At least half the indices see all $q$ retained values: if
$\gamma_j<\tau$, then $\gamma_j\le\tau'$ for the largest
$\tau'\in\Gamma$ below $\tau$, and $2|S_{\tau'}|<|I|$ by minimality.
Hence $2\,|\{j\in I:\gamma_j\ge\tau\}|\ge|I|$, and every such $j$ has
$|\{v\in\Gamma:v\le\gamma_j\}|\ge q$. Therefore
$\sum_{j\in I}|\{v\in\Gamma:v\le\gamma_j\}|\ge q|I|/2$.

For each $v\in\Gamma$ fix an index $\ell_v\in I$ with
$\gamma_{\ell_v}=v$. For $k>\lfloor n/2\rfloor$, $j\in I$, and
$v\in\Gamma$ with $v\le\gamma_j$, the number
\[
 y_k+x_j+v
 =y_k+z_j-y_j+z_{\ell_v+1}-y_{\ell_v+1}-z_{\ell_v}+y_{\ell_v}
\]
belongs to $2Y+2Z-2Y-Z$ and to
$(y_k+x_j,y_k+x_{j+1}]$.
For fixed $k$, these half-open intervals are disjoint as $j$ varies,
and within one interval distinct $v$ give distinct numbers; so each $k$
contributes at least $q|I|/2$ values. They lie strictly inside
$(y_k,z_k)$ because $j+1<k$ and $x_{j+1}<x_k$. Different $k$ therefore
give disjoint contributions. There are at least $n/2$ choices of $k$,
proving the lemma.
The counting uses all of $I$; $I'$ is retained for the next application.
\end{proof}

\begin{lemma}[The quadratic surface]\label{lem:quadric}
Fix $0<a<b<c<d$ and put
$\kappa=(d-a)(c-b)/((b-a)(d-c))>0$.
For $0<u<v$ define
\begin{align*}
 X&=\frac{2+(a+b)(u+v)+2abuv}{(b-a)(v-u)},\\
 Y&=\frac{2+(b+c)(u+v)+2bcuv}{(c-b)(v-u)},\\
 Z&=\frac{2+(c+d)(u+v)+2cduv}{(d-c)(v-u)}.
\end{align*}
These triples lie on the irreducible nondegenerate surface
\begin{equation}\label{eq:quadric}
 Q(X,Y,Z)=XZ+XY+YZ-\kappa Y^2+\kappa+1=0,
\end{equation}
equivalently $(X+Y)(Y+Z)=(\kappa+1)(Y^2-1)$, and the pair $(X,Y)$
already determines $(u,v)$.
\end{lemma}
\begin{proof}
For real parameters $\alpha<\beta$ let $W(\alpha,\beta)$ denote the
displayed fraction with numerator $2+(\alpha+\beta)(u+v)+2\alpha\beta uv$
and denominator $(\beta-\alpha)(v-u)$, so that $X=W(a,b)$, $Y=W(b,c)$,
$Z=W(c,d)$. Since $(\alpha+\beta)(u+v)\mp(\beta-\alpha)(v-u)$ equals
$2(\alpha v+\beta u)$ or $2(\alpha u+\beta v)$, the numerators of
$W\mp1$ factor:
\begin{equation}\label{eq:master}
 W(\alpha,\beta)-1=\frac{2(1+\alpha v)(1+\beta u)}{(\beta-\alpha)(v-u)},
 \qquad
 W(\alpha,\beta)+1=\frac{2(1+\alpha u)(1+\beta v)}{(\beta-\alpha)(v-u)}.
\end{equation}
All displayed factors are positive; in particular $X,Y,Z>1$.
Adding $X+1$ to $Y-1$ and $Y+1$ to $Z-1$, and using
$(c-b)(1+au)+(b-a)(1+cu)=(c-a)(1+bu)$ together with its analogue for
$(b,c,d)$, gives
\[
 X+Y=\frac{2(c-a)(1+bu)(1+bv)}{(b-a)(c-b)(v-u)},
 \qquad
 Y+Z=\frac{2(d-b)(1+cu)(1+cv)}{(c-b)(d-c)(v-u)},
\]
while multiplying the two identities \eqref{eq:master} for $Y$ gives
$Y^2-1=4(1+bu)(1+bv)(1+cu)(1+cv)/((c-b)^2(v-u)^2)$.
Since $(d-a)(c-b)+(b-a)(d-c)=(c-a)(d-b)$, that is,
$\kappa+1=(c-a)(d-b)/((b-a)(d-c))$, the three displays combine to
$(X+Y)(Y+Z)=(\kappa+1)(Y^2-1)$, which expands to \eqref{eq:quadric}.

Dividing the identities \eqref{eq:master} for $Y$ by the formula for
$X+Y$ cancels the common factors:
$(Y+1)/(X+Y)=(b-a)(1+cv)/((c-a)(1+bv))$ and
$(Y-1)/(X+Y)=(b-a)(1+cu)/((c-a)(1+bu))$.
A value of $(1+cv)/(1+bv)$ determines $v$, since cross-multiplying two
solutions gives $(c-b)(v-v')=0$; likewise for $u$. Hence $(X,Y)$
determines $(u,v)$.

As a polynomial in $Z$, $Q$ has coefficient $X+Y$ and constant
term $XY-\kappa Y^2+\kappa+1$. These are coprime, since setting
$X=-Y$ in the latter gives $(\kappa+1)(1-Y^2)$.
Gauss's lemma and linearity in $Z$ prove irreducibility.
The same argument holds over $\mathbb C$.
All three partial derivatives are nonzero polynomials.

On $X+Y\ne0$, the product form solves for $Z$:
\[
 Z=g(X,Y)=-Y+\frac{(\kappa+1)(Y^2-1)}{X+Y},
 \qquad
 H=\kappa Y^2+2\kappa XY-X^2+\kappa+1.
\]
On the open set where the indicated derivatives are nonzero,
direct differentiation gives
\begin{equation}\label{eq:mixedderivative}
 \frac{\partial^2}{\partial X\partial Y}
       \log\left|\frac{g_X}{g_Y}\right|
 =
 -\frac{2\kappa(X^2+2XY-\kappa Y^2+\kappa+1)}{H^2}.
\end{equation}
This rational function cannot vanish identically on a nonempty
open set, since $\kappa>0$.
An additive representation
$g=\psi(\phi_1(X)+\phi_2(Y))$, with real or complex analytic separate
coordinate changes, would make the mixed derivative
identically zero, by the chain rule.
The excluded algebraic curves contain no open surface patch, so
the irreducible surface is nondegenerate as required in
\eqref{eq:SZ}.
The derivative calculation also holds on a complex open set using
a local logarithm, and so excludes the complex analytic alternative
in the stated incidence theorem.
\end{proof}

\begin{proof}[Proof of Theorem \ref{thm:expander}]
Small sets are absorbed by the absolute constant.
For a large $S$, retain a constant fraction of its nonzero elements
having the same sign, and reflect that subset if necessary.
Products of two retained elements are unchanged by this reflection.
It suffices to prove the assertion for a positive set of size $n$;
the resulting quotient set is contained in the original one.
Let $A=\log S=\{\alpha_1<\cdots<\alpha_n\}$ and
$f(t)=\log(1+e^t)$.
Choose $i$ minimizing $\alpha_{i+3}-\alpha_i$ and put
$m=\lfloor n/3\rfloor$.
The intervals
$(\alpha_i+\alpha_{3j},\alpha_{i+3}+\alpha_{3j})$ are disjoint:
consecutive starting points are separated by at least their length.
For each $r=0,1,2$, define
\[
 Y_r=\{f(\alpha_{i+r}+\alpha_{3j}):1\le j\le m\},
 \quad
 Z_r=\{f(\alpha_{i+r+1}+\alpha_{3j}):1\le j\le m\}.
\]
The paired endpoints strictly interleave. Their gaps
\[
 x_{r,j}
 =f(\alpha_{i+r+1}+\alpha_{3j})
       -f(\alpha_{i+r}+\alpha_{3j})
\]
are strictly increasing in $j$, because $f$ is strictly convex.

Start with $I_0=\{1,\ldots,\lfloor m/2\rfloor-1\}$ and apply
Lemma \ref{lem:median} successively for $r=0,1,2$.
This gives nested subsets $I_3\subseteq I_2\subseteq I_1\subseteq I_0$
with $|I_3|\gg n$.
Define
$\Gamma_r=\{x_{r,j+1}-x_{r,j}:j\in I_3\}$.
Since restricting indices can only shrink these gap sets, each
application gives
\begin{equation}\label{eq:squeeze-three}
 |2Y_r+2Z_r-2Y_r-Z_r|\gg n^2|\Gamma_r|.
\end{equation}

Set
$a=e^{\alpha_i}$, $b=e^{\alpha_{i+1}}$,
$c=e^{\alpha_{i+2}}$, $d=e^{\alpha_{i+3}}$.
For each $j\in I_3$, put $u=e^{\alpha_{3j}}$,
$v=e^{\alpha_{3(j+1)}}$. If
$\gamma=x_{0,j+1}-x_{0,j}$, then
\[
 e^\gamma=\frac{(1+bv)(1+au)}{(1+av)(1+bu)}>1,
 \qquad
 \frac{e^\gamma+1}{e^\gamma-1}
 =\frac{2+(a+b)(u+v)+2abuv}{(b-a)(v-u)}.
\]
The second identity is \eqref{eq:master} for the pair $(a,b)$, since
$w\mapsto(w+1)/(w-1)$ is an involution.
The corresponding formulas for $r=1,2$ give $Y,Z$ in
Lemma \ref{lem:quadric}.
Thus the separate maps
$\gamma\mapsto(e^\gamma+1)/(e^\gamma-1)$ put at least $|I_3|$
distinct triples on $Q=0$ in a Cartesian product whose coordinate
sizes are $|\Gamma_0|,|\Gamma_1|,|\Gamma_2|$.
Injectivity here follows from Lemma \ref{lem:quadric}, rather than
from an assumed absence of repeated gaps.
Equation \eqref{eq:SZ} yields
\[
 n\ll |I_3|
 \ll \max_r|\Gamma_r|^{12/7},
 \qquad
 \max_r|\Gamma_r|\gg n^{7/12}.
\]
Choose $r$ with largest gap set and use \eqref{eq:squeeze-three}.
The additive word there is contained in
$2f(\alpha_{i+r}+A)+2f(\alpha_{i+r+1}+A)
-2f(\alpha_{i+r}+A)-f(\alpha_{i+r+1}+A)$.
Exponentiation is injective and turns this into the required
seven-factor quotient set for the two corresponding elements of $S$.
Its cardinality is therefore $\gg n^{31/12}$.
\end{proof}

\subsection*{Use of artificial intelligence}
The initial ideas and the research direction are due to the author.
The majority of the technical work was carried out by large language
model coding agents under the author's direction and review: the main
argument --- the separate-minimizer bridge and its geometric
implementation --- together with its initial audits and machine
certification was developed primarily in sessions of OpenAI Codex,
while the independent line-by-line verification, the simplified proofs
of the present version, most of the accompanying Lean~4 formalization
(the median lemma, the surface identities, and the pinned
construction), the sharpness and pinned additions, the literature
checks, and this manuscript were produced primarily in sessions of
Anthropic Claude Code. All Lean files are kernel-checked against
\texttt{mathlib} and report only the standard axioms; the
exact-arithmetic check scripts are likewise machine-written. The
author takes full responsibility for the correctness and integrity of
this article. All sources, verification files, and their kernel
outputs are available at
\url{https://github.com/Lzp88/dot-product-multibase-improvement}.

\end{document}